\documentclass[12pt]{amsart}
\usepackage{amscd, amssymb, pgf, tikz, hyperref}

\newcommand{\field}[1]{\mathbb{#1}}
\newcommand{\CC}{\field{C}}

\newcommand{\TT}{\field{T}}
\newcommand{\ZZ}{\field{Z}}
\newcommand{\RR}{\field{R}}

\newcommand{\Ee}{\mathcal{E}}
\newcommand{\Gg}{\mathcal{G}}
\newcommand{\Hh}{\mathcal{H}}
\newcommand{\Kk}{\mathcal{K}}
\newcommand{\Ll}{\mathcal{L}}

\newcommand{\Oo}{\mathcal{O}}

\newcommand{\la}{\langle}
\newcommand{\ra}{\rangle}

\newcommand{\ds}{\displaystyle}

\newtheorem{thm}{Theorem}[section]
\newtheorem{cor}[thm]{Corollary}
\newtheorem{lem}[thm]{Lemma}
\newtheorem{prop}[thm]{Proposition}

\theoremstyle{definition}
\newtheorem{dfn}[thm]{Definition}

\theoremstyle{remark}
\newtheorem{rmk}[thm]{Remark}

\newtheorem{example}[thm]{Example}

\newtheorem*{examples*}{Examples}

\numberwithin{equation}{subsection}

\title{Groupoid actions on  $C^*$-correspondences}

\author{Valentin Deaconu}
\address{Valentin Deaconu \\ Department of Mathematics (084)\\ University
of Nevada\\ Reno NV 89557-0084\\ USA} \email{vdeaconu@unr.edu}

\keywords{$C^*$-correspondence; Groupoid action; Groupoid representation; Graph algebra;  Cuntz-Pimsner algebra.}

\subjclass{Primary 46L05.}
\begin{document}
\begin{abstract}
Let the groupoid $G$ with unit space $G^0$ act via a representation $\rho$ on a  $C^*$-correspondence ${\mathcal H}$ over the $C_0(G^0)$-algebra $A$. By the universal property, $G$ acts  on the Cuntz-Pimsner algebra ${\mathcal O}_{\mathcal H}$ which becomes a $C_0(G^0)$-algebra. The action of $G$ commutes with the gauge action on ${\mathcal O}_{{\mathcal H}}$, therefore $G$ acts also on the core algebra ${\mathcal O}_{\mathcal H}^{\mathbb T}$. 

We study the crossed product  ${\mathcal O}_{\mathcal H}\rtimes G$ and the fixed point algebra ${\mathcal O}_{\mathcal H}^G$ and obtain similar results as in \cite{D}, where  $G$ was a group.    Under certain conditions, we prove that ${\mathcal O}_{\mathcal H}\rtimes G\cong \Oo_{\Hh\rtimes G}$, where $\Hh\rtimes G$ is the crossed product $C^*$-correspondence and  that ${\mathcal O}_{\mathcal H}^G\cong{\mathcal O}_\rho$, where ${\mathcal O}_\rho$ is the Doplicher-Roberts algebra defined using intertwiners.

The motivation of this paper comes from groupoid actions on graphs. Suppose $G$ with  compact isotropy acts on a discrete  locally finite graph $E$ with no sources. Since $C^*(G)$ is strongly Morita equivalent to a commutative $C^*$-algebra, we prove that the crossed product $C^*(E)\rtimes G$ is stably isomorphic to a  graph algebra. We illustrate with some examples.
  
\end{abstract}
\maketitle
\section{introduction}

\bigskip

Group actions   provide important bridges between dynamics and $C^*$-algebras via the crossed product construction. In this paper, we consider groupoids acting on fibered spaces and on directed graphs. Under certain conditions, these actions can be extended to the graph $C^*$-correspondence and we obtain interesting results about the associated Cuntz-Pimsner algebras.

 Given a group $G$ acting on a  $C^*$-correspondence ${\mathcal H}$ over a $C^*$-algebra $A$ via $\rho:G\to {\mathcal L}_{\mathbb C}({\mathcal H})$, in \cite{D} we studied the  fixed point algebra ${\mathcal O}_{\mathcal H}^G$ and the crossed product ${\mathcal O}_{\mathcal H}\rtimes G$, with applications to group actions on graphs.  The Doplicher-Roberts algebra ${\mathcal O}_\rho$ was defined  from intertwiners $(\rho^m, \rho^n)$, where $\rho^n=\rho^{\otimes n}$ is the tensor power representation of $G$ on the balanced tensor product ${\mathcal H}^{\otimes n}$.
We proved that in certain cases ${\mathcal O}_\rho$ is isomorphic to ${\mathcal O}_{\mathcal H}^G$ and strongly Morita equivalent to
 ${\mathcal O}_{\mathcal H}\rtimes G$.
 
 In this paper, we  obtain similar results for groupoid actions, under some extra assumptions. We  first recall the machinery associated with groupoid actions on spaces, on other groupoids, on graphs, on $C_0(X)$-algebras and on $C^*$-correspondences.  If $G$ is a locally compact amenable groupoid with Haar system acting on a $C^*$-correspondence $\Hh$ and $J_\Hh$ is the Katsura ideal, we assume that $J_\Hh\rtimes G\cong J_{\Hh\rtimes G}$ in order to obtain the isomorphism ${\mathcal O}_{\mathcal H}\rtimes G\cong \Oo_{\Hh\rtimes G}$ as in \cite{HN}. We illustrate with some examples, including self-similar actions on the path space of a finite graph and actions on Hermitian vector bundles.
 
 As an application, if $G$ with  compact isotropy groups acts on a  discrete  locally finite graph $E$ with no sources, we prove that $C^*(E)\rtimes G$ is stably isomorphic to the $C^*$-algebra of a graph. Since the action of $G$ commutes with the gauge action of $\TT$ on $C^*(E)$, the groupoid $G$ also acts  on the core AF-algebra $C^*(E)^{\mathbb T}$ and $C^*(E)^{\mathbb T}\rtimes G\cong (C^*(E)\rtimes G)^{\mathbb T}$ is an AF-algebra.
 
 \bigskip

\subsection*{Acknowledgements} The author would like to thank Alex Kumjian and Leonard Huang for helpful and illuminating discussions.

\bigskip

\section{Groupoid actions on spaces}
\bigskip

Let's  assume that $G$ is a second countable locally compact Hausdorff  groupoid with unit space $G^0$ and range and source maps $r,s:G\to G^0$. The set of composable pairs is denoted $G^2$ and the set of $g\in G$ with $s(g)=u, r(g)=v$ is denoted $G_u^v$. We first recall the definition of a groupoid action on a space given in \cite{MRW1}:

\begin{dfn}\label{sa}  A topological groupoid $G$ is said to act
(on the left) on a locally compact space $X$, if there are given 
a continuous, open surjection $p : X \rightarrow G^0$,  called the anchor or momentum map,
and a continuous map
\[G\ast X \rightarrow X, \quad\text{write}\quad (g , x)\mapsto
g \cdot x,\]
where
\[G \ast X = \{(g , x)\in G \times X \mid s(g) = p (x)\},\]
that satisfy

\medskip

i) $p (g \cdot x) =r (g)$ for all $(g , x) \in G \ast X,$

\medskip

ii) $(g _2, x) \in G \ast X,\,\, (g_1, g_2)
\in G ^2$ implies $(g _1g _2, x),
(g _1, g _2\cdot x) \in G * X$ and
\[g _1\cdot(g _2\cdot x) = (g _1g _2)
\cdot x,\]

\medskip

iii) $p (x)\cdot x = x$ for all $x\in X$. 
\end{dfn}
We should mention that recently, many authors  assume that $p : X\to G^0$ is not necessarily open (see  \cite{BM} for example).

The action is free if $g\cdot x=x$ for some $x$ implies $g=p(x)$.
The set of fixed points is defined as \[X^G=\{x\in X: g\cdot x=x\;\text{for all}\; g\in G_{p(x)}^{p(x)}\},\] where \[G_u^u=\{g\in G: s(g)=r(g)=u\}\] is  the isotropy group at $u\in G^0$. 

\begin{rmk}
Note that the set $X^G$  is fibered over $G^0$ via the restriction of the map $p$. It is the largest subset of $X$ on which $G$ acts trivially, in the sense that for any $u,v\in G^0$ and for any  $g,h\in G_u^v$, the elements $g,h$ induce the same maps between the fibers $p^{-1}(u)$ and $p^{-1}(v)$.   If $G$ has trivial isotropy, then $X^G=X$.
\end{rmk}
For $x\in X$, its  stabilizer group is
\[G(x)=\{g\in G: g\cdot x=x\},\] which is a subgroup of  $G_u^u$ where $u=p(x)$. The set of orbits \[G\ast x=\{g\cdot x: g\in G, s(g)=p(x)\}\] is denoted by $X/G$.

\begin{rmk}The fibered product $G \ast X$ has a natural
structure of  groupoid, called the semi-direct product or
action groupoid and is denoted  by $G \ltimes X$, where
\[(G \ltimes X)^2 = \{ ((g_1, x_1),(g _2, 
x_2)) \mid \,\,x_1 = g _2\cdot x_2\},\]
with operations
\[(g _1, g_2\cdot x_2)(g _2,  x_2) = 
(g _1g _2, x_2), \;\;
(g,x)^{-1} = (g ^{-1}, g \cdot x).\]
The source and range maps of $G \ltimes X$ are
 \[s(g ,x) = (s(g),x)=
(p( x), x),\quad
r(g ,x) =  (r(g), g\cdot x)=(p(g\cdot  x),g\cdot x),\]
and the unit space $(G\ltimes X)^0$ may be identified with $X$ via the map \[i:X\to G\ltimes X,\;\; i(x)= (p(x),x).\] The projection map\[\pi:G\ltimes X\to G, \;\; \pi(g,x)=g\] is a covering of groupoids, see \cite{DKR}. 
\end{rmk}
\begin{example} A groupoid $G$ with open range and source maps acts on its unit space $G^0$ by $g\cdot s(g)=r(g)$. Notice that $g\cdot u=u$ for all $g\in G_u^u$, in particular $(G^0)^G=G^0$.
\end{example}
If there is only one orbit in $G^0/G$, the groupoid is called transitive. For example, let the group $\mathbb R$ act on $\mathbb R/\mathbb Z$ by translation. Then
the action groupoid $\RR\ltimes\RR/\ZZ$ with unit space $\RR/\ZZ\cong\TT$ and isotropy $\ZZ$ is transitive. 

A transitive groupoid with discrete unit space is of the form $G^0\times K\times G^0$ where $K$ is the isotropy group. 

\bigskip

\section{Groupoid actions on groupoids}

\bigskip

The notion of a groupoid action on another groupoid was defined in \cite{Br} in the algebraic case. In the topological context, we give the following definition (see  \cite{AR}, page 122).
\begin{dfn}\label{gg}
 We say that a topological groupoid $G$ acts on another topological groupoid $H$ if there are a continuous open surjection $p: H\to G^0$ and a continuous map $G\ast H\to H$, write $(g,h)\mapsto g\cdot h$ where
  \[G\ast H=\{(g,h)\in G\times H \mid s(g)=p(h)\}\] 
   such that
   
   \medskip
   
   i)  $p(g\cdot h)=r(g)$  for all $(g,h)\in G\ast H$,
   
   \medskip
 
 ii) $(g_2, h)\in G\ast H, \;(g_1, g_2)\in G^2$ implies $(g_1g_2,h)\in G\ast H$ and \[(g_1g_2)\cdot h=g_1\cdot (g_2\cdot h),\]
 
 \medskip
 
 iii) $(h_1,h_2)\in H^2$ and $(g, h_1h_2)\in G\ast H$ implies $(g,h_1), (g,h_2)\in G\ast H$ and \[g\cdot(h_1h_2)=(g\cdot h_1)(g\cdot h_2),\]
 
 \medskip
 
 iv) $p(h)\cdot h=h$ for all $h\in H$.

\end{dfn}
\begin{rmk}
If $G$ acts on $H$, then in particular $G$ acts on the unit space $H^0$ via the restriction $p_0:=p\!\mid_{H^0}:H^0\to G^0$ and we have $p=p_0\circ r=p_0\circ s$. Using the fact that $h=hs(h)=r(h)h$, we deduce that we also have $s(g\cdot h)=g\cdot s(h)$ and $r(g\cdot h)=g\cdot r(h)$.
\end{rmk}
For example, a $\Gamma$-sheaf  as in \cite{Ku} is given by an \' etale groupoid $\Gamma$ acting on a group bundle $A$ over $\Gamma^0$. In this case $p_0$ is the identity.

\begin{rmk}
If the groupoid $G$ acts on $H$ via $p$, then the triple $(H, G^0,p)$ becomes a continuous field of groupoids in the terminology of \cite{LR}, which determines a continuous $C^*$-bundle. If we do not assume $p$ to be open, we just get an upper semicontinuous $C^*$-bundle. 

When $G^0$ is discrete, the groupoid $H$ is a disjoint union of groupoids $H_u=p^{-1}(u)$ for $u\in G^0$. 
\end{rmk}

The fixed point groupoid $H^G$ is defined as 
\[H^G=\{h\in H:g\cdot h=h\;\text{for all}\; g\in G_{p(h)}^{p(h)}\}.\]
\begin{rmk}
If $G$ is transitive, then $H^G$ is non-empty if and only if $p:H\to G^0$ has invariant sections, i.e. continuous maps $\sigma :G^0\to H^0$ such that  $p(\sigma(u))=u$ for all $u\in G^0$ and such that $\sigma$ commutes with the actions of $G$ on $H$ and on $G^0$, i.e. $g\cdot \sigma(u)=\sigma(g\cdot u)$. 
\end{rmk}
\begin{proof}

Indeed, if $\sigma$ is an invariant section, then for $u\in G^0$ and $h=\sigma(u)\in H^0$ we have $g\cdot h=g\cdot \sigma(u)=\sigma(u)=h$ for all $g\in G_u^u$.
Conversely, given $a\in (H^G)^0$ with $p_0(a)=u$ we define $\sigma:G^0\to H^0$ by $\sigma(v)=g\cdot a$ for some $g\in G_u^v$ and $g\cdot a$ is independent of the choice of $g$.
\end{proof}
The crossed product groupoid $G\ltimes H$ is defined as follows. As a set, $G\ltimes H=G\ast H$ with multiplication
\[(g_1,h_1)(g_2,h_2)=(g_1g_2, (g_2^{-1}\cdot h_1)h_2),\]
when this makes sense, i.e. when $r(g_2)=s(g_1)=p(h_1)$, $s(g_2)=p(h_2)$ and $g_2^{-1}\cdot s(h_1)=r(h_2)$. The inverse is given by \[(g,h)^{-1}=(g^{-1}, g\cdot h^{-1}).\]
It is easy to check that the multiplication is associative and that
\[s(g,h)=(s(g), s(h)),\;\; r(g,h)=(r(g), g\cdot r(h)).\] 
The unit space of $G\ltimes H$ can be identified with $H^0$, and we have an extension of groupoids
\[H\stackrel{i}{\rightarrow} G\ltimes H\stackrel{\pi}{\rightarrow} G,\]
where $i(h)=(p(h), h)$ and $\pi(g, h)=g$.  The map $\pi$ is a fibration of groupoids, see \cite{DKR} or \cite{BM}. Indeed, given $g\in G$ and $u\in H^0=(G\rtimes H)^0$ with $s(g)=p(u)$, choose $h\in H$ with $s(h)=u$. Then $(g,h)\in G\rtimes H$ satisfies $s(g,h)=s(h)=u$ and $\pi(g,h)=g$.

\begin{rmk}
In \cite{DKR} it was proved that a continuous open surjective homomorphism $\pi: G\to H$ of \' etale groupoids with amenable kernel gives rise to a Fell bundle ${\mathcal E}$ over $H$ such that $C^*_r({\mathcal E})$ is isomorphic to $C^*_r(G)$. As was observed in \cite{BM}, the Fell bundle  is saturated only if $\pi$ is a fibration. Unfortunately, this hypothesis was omitted in the main result of \cite{DKR}. The authors of \cite{BM} also remove the condition that the groupoids are \' etale, considering the more general case of  locally Hausdorff locally compact groupoids with Hausdorff unit spaces and with Haar systems. 

In particular, when the groupoid $G$ acts on $H$, we can apply the results of \cite{BM} to the fibration $G\ltimes H\stackrel{\pi}{\rightarrow} G$ and obtain a saturated Fell bundle $\mathcal E$ over $G$ such that \[C^*_r({\mathcal E})\cong C^*_r(G\ltimes H)\cong C^*_r(H)\rtimes G.\]

\end{rmk}
\bigskip

\section{Groupoid actions on   graphs}

\bigskip
 
Inspired from the notion of groupoid actions on groupoids, we define now the concept of groupoid actions on graphs. Let $E=(E^0,E^1,r,s)$ be a topological graph, i.e. $E^0,E^1$ are locally compact spaces with $r,s:E^1\to E^0$ continuous maps and $s$ a local homeomorphism. 
\begin{dfn}\label{action} We say that a topological groupoid $G$ acts on  $E$ if $G$ acts on both spaces $E^0, E^1$  in a compatible way. 
This means that there is a continuous open surjection
$p:E^0\to G^0$ such that $p\circ r=p\circ s:E^1\to G^0$ and there are continuous maps $G*E^0\to E^0,\;\; G*E^1\to E^1$  such that  the conditions in Definition \ref{sa} are satisfied and such that
\[s(g\cdot e)=g\cdot s(e),\;\;r(g\cdot e)=g\cdot r(e)\] for $e\in E^1$ and $g\in G$.

\end{dfn}
\begin{rmk}
Since $p$ and $s$ are open, it follows that $p\circ s: E^1\to G^0$ is open. The action of $G$ can be extended to the set of finite paths $E^*=\bigsqcup_{k\ge 0}E^k$, where $E^k$ is the set of paths of length $k$, by \[g\cdot(e_1e_2\cdots e_k)=(g\cdot e_1)(g\cdot e_2)\cdots (g\cdot e_k)\] and similarly to the set of infinite paths $E^\infty$.
If $G^0$ is discrete, note that since $p\circ r=p\circ s$, the graph $E$ is a union of graphs $E_u$ for $u\in G^0$.
\end{rmk}
\begin{example}\label{33}Let $E$ be the graph 
 
 
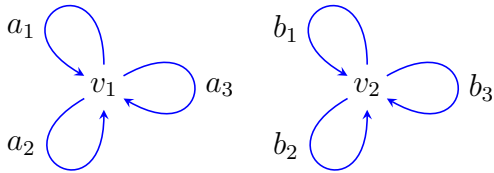
\[
\begin{tikzpicture}[scale=.35]
    \node[inner sep=1.5pt, circle] (v) at (-5,0) {$v_1$};
    \draw[-stealth, semithick,blue] (v) .. controls +(0,5) and +(-4.33,2.5).. (v) node[pos=0.7, left] 
    {\color{black}$a_1$};
    \draw[-stealth, semithick, blue] (v) .. controls +(-4.33,-2.5)and +(0,-5).. (v) node[pos=0.3, left] 
    {\color{black}$a_2$};
    \draw[-stealth, semithick, blue] (v) .. controls +(4.33,2.5) and +(4.33,-2.5) .. (v) node[pos=0.5, right] 
    {\color{black}$a_3$};
     \node[inner sep=1.5pt, circle] (v) at (5,0) {$v_2$};
    \draw[-stealth, semithick,blue] (v) .. controls +(0,5) and +(-4.33,2.5).. (v) node[pos=0.7, left] 
    {\color{black}$b_1$};
    \draw[-stealth, semithick, blue] (v) .. controls +(-4.33,-2.5)and +(0,-5).. (v) node[pos=0.3, left] 
    {\color{black}$b_2$};
    \draw[-stealth, semithick, blue] (v) .. controls +(4.33,2.5) and +(4.33,-2.5) .. (v) node[pos=0.5, right] 
    {\color{black}$b_3$};

\end{tikzpicture}
\]

The transitive groupoid $G$ with unit space $G^0=E^0=\{v_1, v_2\}$ and isotropy $S_3$ acts on $E$ by permutations. In this case $(E^0)^G=E^0$ but $(E^1)^G=\emptyset$.
\end{example}

We recall now the definition of a self-similar action of a groupoid on the path space of a graph as in \cite{LRRW}, which will provide more examples of groupoid actions on graphs.
\begin{dfn}\label{ssa}(Self-similar actions)
Suppose $E$ is a finite graph with no sources. Let $vE^*$ denote the set of finite paths ending at $v$.
We define the graph \[T_E=\bigsqcup_{v\in E^0}vE^*\] to be the  union of rooted trees (also called forest)  with the set of vertices $T^0_E=E^*$ and with the set of edges  \[T^1_E=\{(\mu, \mu e):\mu\in E^*, e\in E^1, s(\mu)=r(e)\}.\]
The set Iso$(E^*)$ of partial isomorphisms $vE^*\to wE^*$ for $v,w\in E^0$ becomes a groupoid.
A {\em self-similar action} of a groupoid $G$ with $G^0=E^0$ on $E^*$ is given by a homomorphism $G\to$ Iso$(E^*)$ such that for every $g\in G$ and $e\in s(g)E^1$ there exists a unique $h\in G$ denoted also by $g\!\mid_e$ such that
\[g\cdot(e\mu)=(g\cdot e)(h\cdot \mu)\;\;\text{for all}\;\; \mu\in s(e)E^*.\]

\end{dfn}

\begin{prop} A self-similar action of a groupoid $G$  on the path space $E^*$ of a finite graph $E$ as above determines an action of $G$ on the graph $T_E$ as in Definition \ref{action}.
\end{prop}
\begin{proof}
Indeed, the vertex space $T_E^0$ is fibered over $G^0=E^0$ via the map $\mu\mapsto r(\mu)$. For $(\mu, \mu e)\in T^1_E$ we set $s(\mu, \mu e)=\mu e$ and $r(\mu, \mu e)=\mu$. Since   $r(\mu e)=r(\mu)$, the edge space $T^1_E$ is also fibered over $G^0$. The action of $G$ on $T^1_E$ is given by
\[g\cdot (\mu, \mu e)=(g\cdot \mu, g\cdot (\mu e))\;\;\text{when}\;\; s(g)=r(\mu).\]
Since
\[s(g\cdot(\mu, \mu e))=s(g\cdot \mu, g\cdot (\mu e))=g\cdot (\mu e)=g\cdot s(\mu, \mu e),\]
\[r(g\cdot(\mu, \mu e))=r(g\cdot \mu, g\cdot (\mu e))=g\cdot \mu =g\cdot r(\mu, \mu e),\]
the  actions on $T_E^0$ and $T_E^1$ are compatible.

\end{proof}
\begin{example}\label{forrest} Let $E$ be the graph 
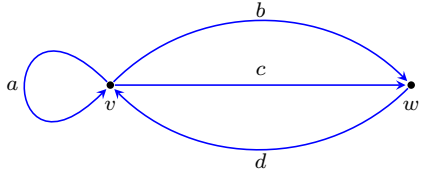
\[\begin{tikzpicture}[shorten >=0.4pt,>=stealth, semithick]
\renewcommand{\ss}{\scriptstyle}
\node[inner sep=1.0pt, circle, fill=black]  (v) at (-2,0) {};
\node[below] at (v.south)  {$\ss v$};
\node[inner sep=1.0pt, circle, fill=black]  (w) at (2,0) {};
\node[below] at (w.south)  {$\ss w$};
\draw[->, blue] (v) to   (w);
\node at (0,0.2){$\ss c$};
\draw[->, blue] (v) to [out=45, in=135]  (w);
\node at (0,1){$\ss b$};
\draw[->, blue] (w) to [out=-135, in=-45]  (v);
\node at (0,-1) {$\ss d$};
\draw[->, blue] (v) .. controls (-3.5,1.5) and (-3.5, -1.5) .. (v);
\node at (-3.3,0) {$\ss a$};
\end{tikzpicture}
\]

with forest 

\hspace{55mm} $T_E$
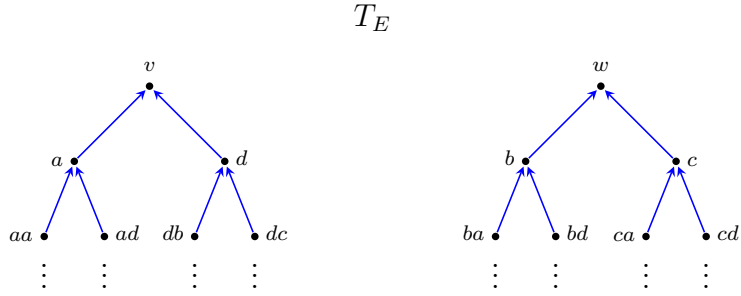
\[
\begin{tikzpicture}[shorten >=0.4pt,>=stealth, semithick]
\renewcommand{\ss}{\scriptstyle}
\node[inner sep=1.0pt, circle, fill=black]  (v) at (-4,4) {};
\node[above] at (v.north)  {$\ss v$};
\node[inner sep=1.0pt, circle, fill=black]  (w) at (2,4) {};
\node[above] at (w.north)  {$\ss w$};
\node[inner sep=1.0pt, circle, fill=black] (a) at (-5,3){};
\node[left] at (-5,3)  {$\ss a$};
\node[inner sep=1.0pt, circle, fill=black] (d) at (-3,3){};
\node[right] at (-3,3.05)  {$\ss d$};
\draw[->, blue] (a) to  (v);
\draw[-> , blue] (d) to  (v);
\node[inner sep=1.0pt, circle, fill=black] (aa) at (-5.4,2){};
\node[inner sep=1.0pt, circle, fill=black] (ad) at (-4.6,2){};
\node[left] at (-5.4,2)  {$\ss aa$};
\node[below] at (-5.4,2) {$\vdots$};
\draw[-> , blue] (aa) to   (a);
\node[right] at (-4.6,2.05)  {$\ss ad$};
\node[below] at (-4.6,2) {$\vdots$};
\draw[-> , blue] (ad) to  (a);
\node[inner sep=1.0pt, circle, fill=black] (db) at (-3.4,2){};
\node[inner sep=1.0pt, circle, fill=black] (dc) at (-2.6,2){};
\node[left] at (-3.4,2.05) {$\ss db$};
\node[below] at (-3.4,2) {$\vdots$};
\node[right] at (-2.6,2.05) {$\ss dc$};
\node[below] at (-2.6,2) {$\vdots$};
\draw[-> , blue] (dc) to   (d);
\draw[-> , blue] (db) to   (d);

\node[inner sep=1.0pt, circle, fill=black] (b) at (1,3){};
\node[left] at (1,3.05)  {$\ss b$};

\node[inner sep=1.0pt, circle, fill=black] (c) at (3,3){};
\node[right] at (3,3)  {$\ss c$};
\draw[-> , blue] (b) to  (w);
\draw[-> , blue] (c) to  (w);
\node[inner sep=1.0pt, circle, fill=black] (ba) at (0.6,2){};
\node[inner sep=1.0pt, circle, fill=black] (bd) at (1.4,2){};

\draw[-> , blue] (ba) to   (b);

\draw[-> , blue] (bd) to  (b);
\node[inner sep=1.0pt, circle, fill=black] (ca) at (2.6,2){};
\node[inner sep=1.0pt, circle, fill=black] (cd) at (3.4,2){};
\draw[-> , blue] (ca) to   (c);
\draw[-> , blue] (cd) to   (c);
\node[left] at (0.6,2.05)  {$\ss ba$};
\node[below] at (0.6,2) {$\vdots$};
\node[left] at (2.6,2)  {$\ss ca$};
\node[below] at (2.6,2) {$\vdots$};
\node[right] at (1.4,2.05) {$\ss bd$};
\node[below] at (1.4,2) {$\vdots$};
\node[right] at (3.4,2.05) {$\ss cd$};
\node[below] at (3.4,2) {$\vdots$};
\end{tikzpicture}
\]
 Consider the groupoid $G$ with unit space $G^0=\{v,w\}$ and generators $g,h$ where $g\in G_v^w, h\in G_w^v$ such that
 \[g\cdot a=c,\;\; g\!\mid_a=v,\;\; g\cdot d=b,\;\; g\!\mid_d=h,\]
\[h\cdot b=a,\;\; h\!\mid_b=v, \;\; h\cdot c=d, \;\; h\!\mid_c=g.\] 
These conditions are also presented as
\[g\cdot a\mu=c\mu, \;  \; g\cdot d\mu =b(h\cdot\mu), \; h\cdot(b\mu)=a\mu, \; h\cdot c\mu=d(g\cdot \mu),\]
and they  determine uniquely an action of  $G=\la g, h\ra$ on the graph $T_E$. 

\end{example}

\begin{prop} Suppose $E$ is a discrete locally finite graph with no sources. If the groupoid $G$ acts on   $E$  via a map $p:E^0\to G^0$, then $G$ acts on the path groupoid ${\mathcal G}_E$ of $E$. 
\end{prop}
\begin{proof}
Recall that 
\[{\mathcal G}_E=\{(x,k,y)\in E^\infty\times\ZZ\times E^\infty: \exists N\;\text{with}\; x_i=y_{i-k}\;\text{for}\; i\ge N\}\]
with operations \[(x,m,y)\cdot (y,n,z)=(x,m+n,z),\; \;(x,k,y)^{-1}=(y,-k,x)\]
and that the unit space ${\mathcal G}_E^0$ can be identified with $E^\infty$.
There is a map $\pi_0:E^\infty\to G^0$ induced by $p:E^0\to G^0$ and for $g\in G$ and $(x,k,y)\in {\mathcal G}_E$ with $s(g)=\pi_0(x)=\pi_0(y)$  we define \[g\cdot (x,k,y)=(g\cdot x, k, g\cdot y).\] It is routine to check that this action satisfies all the properties of Definition \ref{gg}. Continuity is proved using cylinder sets.
In particular, since $E^0$ and $G^0$ are discrete, the groupoid ${\mathcal G}_E$ is a disjoint union of groupoids ${\mathcal G}_{E_u}$ for $u\in G^0$.
\end{proof}
\bigskip

\section{Groupoid actions on $C^*$-algebras and $C^*$-correspondences}
\bigskip

Let $X$ be a locally compact Hausdorff space. Recall that a  $C^*$-algebra $A$ is a $C_0(X)-$algebra if there is a  non-degenerate homomorphism $\varphi: C_0(X)\to ZM(A)$, where $ZM(A)$ denotes the center of the multiplier algebra of $A$. This means  that  $\overline{\varphi (C_0(X))A}=A$, i.e. there is an approximate unit $\{f_i\}$ in $C_0(X)$ such that $\ds\lim_i\|\varphi(f_i)a-a\|=0$ for all $a\in A$. Sometimes we write $f a$ instead of $\varphi(f)a$. 

Given a $C_0(X)$-algebra $A$, for each $x\in X$ we can define the fiber $A_x$ as $A/\overline{I_xA}$ where \[I_x=\{f\in C_0(X):f(x)=0\}.\] Denote the quotient map $A\to A_x$ by $\pi_x$. It is known that the canonical map \[ A\to \prod_{x\in X}A_x, \;\; a\mapsto(\pi_x(a))_{x\in X}\] is injective and that the fibers $A_x$ give rise to an upper semicontinuous $C^*$-bundle $\mathcal A$ over $X$ such that $A\cong \Gamma_0(\mathcal A)$, the  algebra of continuous sections vanishing at infinity, see Theorem C.26 in \cite{W}.

\begin{dfn}
We say that the topological groupoid $G$ with unit space $G^0$ acts on a $C^*$-algebra $A$ if $A$ is a $C_0(G^0)$-algebra and for each $g\in G$ there is a $\ast$-isomorphism $\alpha_g:A_{s(g)}\to A_{r(g)}$ such that if $(g_1,g_2)\in G^2$ we have $\alpha_{g_1g_2}=\alpha_{g_1}\circ\alpha_{g_2}$ and for a fixed $a$, the map $g\mapsto \alpha_g(a)$ is norm continuous. We also write $g\cdot a$ for $\alpha_g(a)$.
\end{dfn}
\begin{rmk}
An action of a groupoid $G$ on a $C^*$-algebra $A$ can be understood as an isomorphism of $C_0(G)$-algebras $\alpha:s^*A\to r^*A$ such that $\alpha_{g_1g_2}=\alpha_{g_1}\circ\alpha_{g_2}$ for all $(g_1,g_2)\in G^2$. 
Here the pull backs $s^*A$ and $r^*A$ become $C_0(G)$-algebras such that $(s^*A)_g=A_{s(g)}$ and $(r^*A)_g=A_{r(g)}$. It can also be understood as a functor from the small category $G$ to the category of $C^*$-algebras as in \cite{M}.
\end{rmk}
\begin{rmk}
If the groupoid $G$ with Haar system $\lambda$ acts on $A$, then one can define the crossed product $A\rtimes G$ and the reduced crossed product $A\rtimes_r G$ by completing the $*$-algebra $\Gamma_c(G, r^*A)$ of continuous sections with compact support in the appropriate norms (see \cite{MW} for example). 

Recall that for $f, f'\in \Gamma_c(G,r^*A)$, we define
\[f\ast f'(g):=\int_Gf(h)\alpha_h(f'(h^{-1}g))d\lambda^{r(g)}(h),\;\; f^*(g)=\alpha_g(f(g^{-1})^*).\]
The fixed point algebra $A^G$ is defined as the $C_0(G^0)$-algebra with fibers \[A_x^G=\{a\in A_x: g\cdot a=a\;\text{for all}\; g\in G_x^x\}.\]
\end{rmk}
\begin{example} If the groupoid $G$ acts on the locally compact space $X$, then $C_0(X)$ becomes in a natural way a $C_0(G^0)$-algebra and $G$ acts on $C_0(X)$ by $(g\cdot f)(x)=f(g^{-1}\cdot x)$ such that $C_0(X)\rtimes G\cong C^*(G\ltimes X)$ and $C_0(X)^G=C_0(X/G)$.
\end{example}
\begin{example}
Groupoid actions on elementary $C^*$-bundles over $G^0$ satisfying Fell's condition appear in the context of defining the Brauer group Br$(G)$, see \cite{KMRW}.

\end{example}
\begin{example} If the groupoid $G$ with discrete unit space acts on the groupoid $H$, then $G$ acts on $C^*(H)$, which is a $C_0(G^0)$-algebra in a natural way.
\end{example}
Indeed, the map $p_0:H^0\to G^0$ determines a homomorphism $\varphi:C_0(G^0)\to ZM(C^*(H))$ and the fibers of $C^*(H)$ are $C^*(H_u)$ for $u\in G^0$. 

\bigskip
Let $A$ be a $C_0(X)$-algebra and let ${\mathcal H}$ be a Hilbert $A$-module. We define the fibers ${\mathcal H}_x:={\mathcal H}\otimes_AA_x$, for each $x\in X$. Then $\Hh_x$ becomes a Hilbert $A_x$-module with the usual operations. The  set Iso$({\mathcal H})$ of ${\mathbb C}$-linear isomorphisms ${\mathcal H}_x\to {\mathcal H}_y$ becomes a groupoid with unit space $X$, sometimes called the frame groupoid. 
\begin{rmk}
For $T\in{\mathcal L}_A({\mathcal H})$, let $T_x\in{\mathcal L}_{A_x}({\mathcal H}_x)$ be $T\otimes 1_{A_x}$, where $1_{A_x}$ is the identity map. This gives a map ${\mathcal L}_A({\mathcal H})\to {\mathcal L}_{A_x}({\mathcal H}_x)$ and $T$ is totally determined by the family $(T_x)_{x\in X}$. The identification $\Hh=\Hh A$ defines a homomorphism $C_0(X)\to Z \Ll_A(\Hh)$ which takes a function $f\in C_0(X)$ into the operator $\xi\mapsto \xi f$. Since $\overline{\Hh C_0(X)}=\Hh$, this homomorphism induces a structure of $C_0(X)$-algebra on $\Kk_A(\Hh)$ with fibers ${\mathcal K}_{A_x}({\mathcal H}_x)$. In general, $\overline{C_0(X)\Ll_A(\Hh)}\neq \Ll_A(\Hh)$, so that ${\mathcal L}_A(\mathcal H)$ is not a $C_0(X)$-algebra. 
\end{rmk}

\begin{dfn}
Let $A, B$ be  $C_0(X)$-algebras and let ${\mathcal H}$ be a Hilbert $A$-module. We say that ${\mathcal H}$ is a $B$--$A$ $C^*$-correspondence  if there is a $\ast$-homomorphism $\phi: B\to \Ll_A(\Hh)$ such that 
\[\phi(fb)\xi=\phi(b)(\xi f)\;\;\text{for all}\;\;f\in C_0(X),\; b\in B, \; \xi\in \Hh.\]
Sometimes we write $b\xi$ for $\phi(b)\xi$.
\end{dfn}
\begin{rmk}
Since $\phi(\overline{I_xB})\subset \Hh(\overline{I_xA})$, the homomorphism $\phi$ decomposes into a family of homomorphisms $\phi_x:B_x\to\Ll_{A_x}(\Hh_x)$ and each fiber ${\mathcal H}_x$ becomes a $B_x$--$A_x$ $C^*$-correspondence, see Definition 4.1 in  \cite{L2}. 
\end{rmk}
\begin{dfn}\label{gact}
Let $G$ be a topological groupoid  acting on the $C_0(G^0)$-algebras $A$ and $ B$. We say
that $G$ acts on the $B$--$A$ $C^*$-correspondence ${\mathcal H}$  if there is a  homomorphism $\rho:G\to \;\;$Iso$(\mathcal H)$   such that  $g\mapsto\rho(g)\xi$ is norm continuous  and such that the following compatibility conditions are satisfied:
 \[ \langle \rho(g)\xi,\rho(g)\eta\rangle_{r(g)}=g\cdot\langle\xi,\eta\rangle_{s(g)} ,\] for $\xi, \eta\in {\mathcal H}_{s(g)}$ and
 \[ \rho(g)(\xi a)=(\rho(g)\xi)(g\cdot a),\;\; \rho(g)(b\xi)=(g\cdot b)(\rho(g)\xi)\]
  for $\xi \in{\mathcal H}_{s(g)},\; a\in A_{s(g)}$ and $ b\in B_{s(g)}$.
\end{dfn}
\begin{rmk}
 The action of $G$ on $\Hh$ induces an action of $G$  on ${\mathcal L}_A(\mathcal H)$ by \[(g\cdot T)(\xi)=\rho(g)T(\rho(g^{-1})\xi)\] where $T\in \Ll_{A_{s(g)}}(\Hh_{s(g)})$ and on ${\mathcal K}_A(\mathcal H)$ via \[g\cdot \theta_{\xi, \eta}=\theta_{\rho(g)\xi, \rho(g)\eta},\] where for $\xi, \eta\in \Hh_{s(g)}$ we have $\theta_{\xi, \eta}(\zeta)=\xi\la \eta, \zeta\ra$.  
 
It is easy to check that the  left multiplication $\phi: B\to {\mathcal L}_A(\mathcal H)$ is $G$-equivariant. Indeed, 
 \[(g\cdot \phi(b))(\xi)=\rho(g)(\phi(b)(\rho(g^{-1})\xi)=\phi(g\cdot b)(\rho(g)\rho(g^{-1})\xi)=\phi(g\cdot b)(\xi).\]
 \end{rmk}

 \begin{example} If $G$ is a locally compact groupoid with Haar system, then $G$ acts on the  $C^*$-correspondence  over $C_0(G^0)$ with fibers $L^2(G^u,\lambda^u)$  via the left regular representation.
 \end{example}
 
 \bigskip

 \section{Crossed products and Cuntz-Pimsner algebras}
 
 \bigskip
 
In this section we assume that $G$ is a locally compact groupoid with Haar system. We make some extra assumptions in order to obtain similar results as in \cite{HN} for the case of amenable groupoid actions on $C^*$-correspondences.
 
 \begin{dfn}Suppose the locally compact groupoid $G$ with Haar system acts on the $C^*$-correspondence  ${\mathcal H}$  as in Definition \ref{gact}.
 The  crossed product  ${\mathcal H}\rtimes G$ is defined as a completion of $\Gamma_c(G,r^*{\mathcal H})$
which becomes a $(B\rtimes G)$--$(A\rtimes G)$\; $C^*$-correspondence   using the operations
  \[\langle \xi, \eta\rangle(g)=\int_G h\cdot\langle \xi(h^{-1}),\eta(h^{-1}g)\rangle_{s(h)} d\lambda^{r(g)}(h),\]
  \[(\xi f_1)(g)=\int_G\xi(h)(h\cdot f_1(h^{-1}g))d\lambda^{r(g)}(h),\]
  \[(f_2\xi)(g)=\int_Gf_2(h)\rho(h)\xi(h^{-1}g))d\lambda^{r(g)}(h),\]
  where $\xi,\eta\in \Gamma_c(G,r^*{\mathcal H}), f_1\in \Gamma_c(G,r^*A)$ and $ f_2\in \Gamma_c(G,r^*B)$. 
\end{dfn}

\begin{lem}
If the groupoid $G$ acts on a $C^*$-correspondence  $\Hh$ over the $C_0(G^0)$-algebra $A$, then there is an isomorphism \[\kappa:\Kk_{A\rtimes G}(\Hh\rtimes G)\to \Kk_A(\Hh)\rtimes G\]
given by
\[\kappa(\theta_{\xi, \eta})(g)=\int_G\theta_{\xi(h),\rho(h)\eta(h^{-1}g)}d\lambda^{r(g)}(h).\]
\end{lem}

\begin{thm}
Let $G$ be a  locally compact groupoid with Haar system that acts on a $C^*$-correspondence ${\mathcal H}$ over the $C_0(G^0)$-algebra $A$. 

Then the Katsura ideal $J_{\mathcal H}$ is $G$-invariant and  we get an action of $G$ on  the Cuntz-Pimsner algebra ${\mathcal O}_{\mathcal H}$, which becomes a $C_0(G^0)$-algebra with fibers ${\mathcal O}_{{\mathcal H}_x}$ where $x\in G^0$.

\end{thm}
\begin{proof}
If $\phi:A\to {\mathcal L}_A(\mathcal H)$ defines the left action, recall that the Katsura ideal  $J_{\mathcal H}$ is equal to $\phi^{-1}({\mathcal K}_A(\mathcal H))\cap (\ker \phi)^\perp$. Since $\phi$ is $G$-equivariant, both ideals $\phi^{-1}({\mathcal K}_A(\mathcal H))$ and $(\ker \phi)^\perp$ are $G$-invariant. 

A representation of ${\mathcal H}$ in a $C_0(G^0)$-algebra $B$ is a pair $(\pi, t)$ where $\pi:A\to B$ is a $C_0(G^0)$-homomorphism and $t:\mathcal H\to B$ is linear such that
 we have
\[t(a\cdot \xi)=\pi(a)t(\xi),\;\; \pi(\la \xi, \eta\ra)=t(\xi)^*t(\eta)\]
for all $a\in A$ and $\xi, \eta\in\mathcal H$. 

Moreover, if $t^{(1)}:{\mathcal K}_A(\mathcal H)\to B$ is given by
\[t^{(1)}(\theta_{\xi,\eta})=t(\xi)t(\eta)^*,\] then $(\pi, t)$ is covariant if
\[t^{(1)}(\phi(a))=\pi(a)\;\;\text{ for all }\;\; a\in J_{\mathcal H}.\]
The Cuntz-Pimsner algebra ${\mathcal O}_{\mathcal H}$ is generated by a universal covariant representation $(\pi_A, t_{\mathcal H})$ in the sense that for any other covariant representation $(\pi, t)$ in a $C_0(G^0)$-algebra $B$ there is a unique $*$-homomorphism $\pi \times t: {\mathcal O}_{\mathcal H}\to B$ such that \[(\pi\times t)\circ \pi_A=\pi \;\;\text{and}\;\;  (\pi\times t)\circ t_{\mathcal H}=t.\] 

Since $A$ is a $C_0(G^0)$-algebra and each ${\mathcal H}_x$ is a  $C^*$-correspondence over $A_x$, it follows that  ${\mathcal O}_{\mathcal H}$ becomes a $C_0(G^0)$-algebra with fibers ${\mathcal O}_{{\mathcal H}_x}$. The action of $G$ on the Cuntz-Pimsner algebra ${\mathcal O}_{\mathcal H}$ is defined using the universal property and is given by
\[g\cdot \pi_A(a)=\pi_A(g\cdot a)\;\text{for}\; a\in A_{s(g)},\]
\[ g\cdot t_{\mathcal H}(\xi)=t_{\mathcal H}(\rho(g)\xi)\;\text{for}\; \xi\in \Hh_{s(g)}.\]

In particular, the ideal $J_{\mathcal H}$ is a $C_0(G^0)$-algebra with fibers $J_{{\mathcal H}_x}$.

\end{proof}

\begin{thm}
Let $G$ be a locally compact amenable   groupoid with Haar system that acts on a $C^*$-correspondence ${\mathcal H}$ over the $C_0(G^0)$-algebra $A$. Assume that $J_{\Hh\rtimes G}\cong J_{\Hh}\rtimes G$. Then there are maps 
\[\mu: A\rtimes G\to {\mathcal O}_{\Hh}\rtimes G, \;\mu(f)(g)=\pi_A(f(g))\;\text{for}\; f\in \Gamma_c(G,r^*A)\] 
and 
\[\tau:\Gamma_c(G,r^*\Hh)\to {\mathcal O}_{\Hh}\rtimes G,\; \tau(\xi)(g)=t_{\Hh}(\xi(g))\]
which induce an isomorphism 
\[{\mathcal O}_{\Hh\rtimes G}\cong {\mathcal O}_{\Hh}\rtimes G.\]
\end{thm}
\begin{proof}
One verifies that $\tau(\xi)^*\tau(\eta)=\mu(\la \xi, \eta\ra)$, that $\tau$ exends to $\Hh\rtimes G$ and  that $\tau(f\xi)=\mu(f)\tau(\xi)$. Since $J_{\Hh\rtimes G}\cong J_{\Hh}\rtimes G$, it follows as in  Corollary 2.9 of \cite{HN} that $( \mu , \tau)$ is covariant. The rest of the proof is using the fact that the images of $\mu$ and $\tau$ generate ${\mathcal O}_{\Hh}\rtimes G$, the   injectivity of $\mu$, and the existence of the gauge action, see Theorem 6.4 in \cite{K}.
\end{proof}
\begin{example}
Let  $G$ be a  groupoid with compact unit space $X$. Then $G$ acts on $C(X)$ in the usual way, $(g f)(x)=f(g^{-1}x)$. A (locally trivial) complex vector bundle $p:{\mathcal E}\to X$ is called a $G$-bundle if $G$ acts on $\Ee$ by linear maps $\Ee_{s(g)}\to \Ee_{r(g)}$, where $\Ee_x=p^{-1}(x)$. The groupoid $G$ acts on $\Gamma(\Ee)$, the space of continuous sections $\xi:X\to \Ee$, by 
\[(g\xi)(x)=g(\xi(g^{-1}x)).\]  If $\Ee$ is Hermitian and $G$ acts by isometries, then the space $\Gamma(\Ee)$ has a natural structure of $C^*$-correspondence over $C(X)$, where the left and right multiplications are given by \[(f\xi)(x)=(\xi f)(x)=f(x)\xi(x).\]  Since 
\[g(f\xi)(x)=g(f\xi)(g^{-1}x)=g(f(g^{-1}x)\xi(g^{-1}x))=\]\[=f(g^{-1}x)g\xi(g^{-1}x)=(gf)(g\xi)(x),\]
the groupoid $G$ acts on the Cuntz-Pimsner algebra $\mathcal O_{\Gamma(\Ee)}$ which is a continuous field of Cuntz algebras (see \cite{V}).
It would be interesting to study $\mathcal O_{\Gamma(\Ee)}\rtimes G$ in some particular cases.

 \end{example}
\begin{cor}
If the groupoid $G$ acts on a graph $E$ as in Definition \ref{action}, then $G$ acts on the graph $C^*$-correspondence $\Hh_E$ and on  the graph algebra $C^*(E)={\mathcal O}_{\Hh_E}$ which becomes a $C_0(G^0)$-algebra. Since the gauge action on $C^*(E)$ commutes with the action of $G$, we also get an action of $G$ on the core algebra $C^*(E)^{\TT}$.
\end{cor}
\begin{example}
The transitive groupoid $G$ in Example \ref{33} acts on the graph $C^*$-correspondence $\Hh=\CC^3\oplus \CC^3$ over $A=\CC\oplus \CC$. It follows that $G$ acts on the graph algebra ${\mathcal O}_{\Hh}\cong {\mathcal O}_3\oplus{\mathcal O}_3$ and on ${\mathcal O}_{\Hh}^{\TT}\cong M_{3^\infty}\oplus M_{3^\infty}$. The fixed point algebra ${\mathcal O}_{\Hh}^G$ is isomorphic whith ${\mathcal O}_3^{S_3}\oplus {\mathcal O}_3^{S_3}$.
\end{example}
\begin{example}
A self-similar action  as in Definition \ref{ssa} determines an action of $G$ on the graph $C^*$-algebra $C^*(T_E)$. Since $C^*(T_E)$ is strongly Morita equivalent with $C_0(\partial T_E)$ in an equivariant way (see section 4 in \cite{KP}), it follows that $C^*(T_E)\rtimes G$ is strongly Morita equivalent with $C_0(\partial T_E)\rtimes G$. Note that $T_E$ is a union of trees which in general are not the universal cover of the graph $E$.
\end{example}

\bigskip

 \section{Doplicher-Roberts algebras}

\bigskip

The Doplicher-Roberts algebras (denoted  by ${\mathcal O}_G$ in \cite{DR1}) were introduced to construct a new duality theory for compact Lie groups $G$ which strengthens the Tannaka-Krein duality. Let ${\mathcal T}_G$ denote the representation category whose objects are tensor powers of the  $n$-dimensional representation $\rho$ of $G$ defined by the inclusion $G\subseteq U(n)$ in some unitary group $U(n)$ and whose arrows are the intertwiners. The  $C^*$-algebra ${\mathcal O}_G$ is identified in \cite{DR1} with the fixed point algebra ${\mathcal O}_n^G$, where ${\mathcal O}_n$ is the Cuntz algebra. If $\sigma_G$ denotes the restriction to ${\mathcal O}_G$ of the canonical endomorphism of the Cuntz algebra, then ${\mathcal T}_G$ can be reconstructed from the pair $({\mathcal O}_G,\sigma_G)$. Subsequently, Doplicher-Roberts algebras were associated to any object $\rho$ in a strict tensor $C^*$-category, see \cite {DR2}.

Consider now a groupoid $G$   acting on a $C^*$-correspondence ${\mathcal H}$ over the $C_0(G^0)$-algebra $A$ as in Definition \ref{gact} via the homomorphism $\rho:G\to$ Iso$({\mathcal H})$. Since the balanced tensor power $\mathcal H^{\otimes n}$ is fibered over $G^0$ with fibers $\mathcal H_x^{\otimes n}$, it follows that $\Kk_A({\mathcal H}^{\otimes n}, {\mathcal H}^{\otimes m})$ has fibers $\Kk_{A_x}(\Hh_x^{\otimes n},\Hh_x^{\otimes m})$.

Consider the tensor powers $\rho^n:G\to$ Iso$({\mathcal H}^{\otimes n})$ and define the set of intertwiners $(\rho^m, \rho^n)$ with fibers
\[(\rho^m,\rho^n)_x=\{T\in\Kk_{A_x}({\mathcal H}^{\otimes n}_x, {\mathcal H}^{\otimes m}_x)\;\mid \;  T\rho^n(g)=\rho^m(g)T\}.\]
We identify $(\rho^m,\rho^n)$ with a subset of $(\rho^{m+r},\rho^{n+r})$ via $T\mapsto T\otimes I_r$, where $I_r:{\mathcal H}^{\otimes r}\to {\mathcal H}^{\otimes r}$ is the identity map.
After this identification, it follows that the linear span ${}^0{\mathcal O}_\rho$ of $\displaystyle \bigcup_{m,n\ge 0}(\rho^m, \rho^n)$ has a natural multiplication  given by composition: if $S\in (\rho^m,\rho^n)$ and $T\in (\rho^p,\rho^q)$, then the product $ST$ is
 \[(S\otimes I_{p-n})\circ T\in (\rho^{m+p-n},\rho^q) \;\text{if}\; p\ge n,\]
 or
 \[S\circ(T\otimes I_{n-p})\in(\rho^m,\rho^{q+n-p}) \;\text{if}\; p<n.\]
 The adjoint of $T\in(\rho^m,\rho^n)$ is $T^*\in (\rho^n,\rho^m)$.

From Theorem 4.2 in \cite{DR2}, it follows that the $C^*$-closure of ${}^0{\mathcal O}_\rho$  is well defined, obtaining the Doplicher-Roberts algebra ${\mathcal O}_\rho$ associated to the 
homomorphism $\rho:G\to$ Iso$({\mathcal H})$. 
\begin{thm}
Let ${\mathcal H}$ be a full and finite projective $C^*$-correspondence over a $C_0(X)$-algebra $A$  and assume that the left multiplication $A\to {\mathcal L}_A(\mathcal H)$ is injective. If $G$ is a  groupoid with $G^0=X$ acting on $A$ and on $\mathcal H$ via $\rho:G\to$  Iso$({\mathcal H})$ as in Definition \ref{gact}, then the Doplicher-Roberts algebra ${\mathcal O}_\rho$   is isomorphic to the fixed point algebra ${\mathcal O}_{\mathcal H}^G$.
\end{thm}
\begin{proof}
It is known that since $\Hh$ is finite projective, we have ${\mathcal L}_A({\mathcal H})\cong{\mathcal K}_A({\mathcal H})$. Moreover, the Cuntz-Pimsner algebra ${\mathcal O}_{\mathcal H}$ is isomorphic to the $C^*$-algebra generated by the span of 
$\ds \bigcup_{m,n\ge 0}{\mathcal K}_A({\mathcal H}^{\otimes n}, {\mathcal H}^{\otimes m})$
after we identify $T$ with $T\otimes I$, (see Proposition 2.5 in \cite{KPW}). This isomorphism preserves the $C_0(X)$-algebra structures.

The groupoid $G$ acts on ${\mathcal K}_A({\mathcal H}^{\otimes n}, {\mathcal H}^{\otimes m})$  by \[(g\cdot T)(\xi)=\rho^m(g)T(\rho^n(g^{-1})\xi)\] and the fixed point algebra is $(\rho^m, \rho^n)$. Indeed, for a fixed $x\in X$ and $g\in G_x^x$ we have $g\cdot T=T$ if and only if
$T\rho^n(g)=\rho^m(g)T$.

It follows that ${}^0{\mathcal O}_\rho\subseteq {\mathcal O}_{\mathcal H}$ and  that ${\mathcal O}_\rho$ is isomorphic to ${\mathcal O}_{\mathcal H}^G$.
\end{proof}

\begin{cor} Let $E$ be a topological graph such that ${\mathcal H}_E$ is full and finite projective  and the left multiplication of $C_0(E^0)$ is injective. If $G$ is a 
groupoid  acting on $E$ as in Definition \ref{action}
inducing a homomorphism  $\rho:G\to$ Iso$({\mathcal H}_E)$,
then ${\mathcal O}_\rho\cong C^*(E)^G$.
\end{cor}
\begin{example}
The Doplicher-Roberts algebra for the  groupoid action in Example \ref{33} is isomorphic with ${\mathcal O}_3^{S_3}\oplus {\mathcal O}_3^{S_3}$.
\end{example}

\bigskip

\section{Compact isotropy groupoid actions on graphs}

\bigskip

Suppose  the groupoid $G$ with  compact isotropy groups acts on a discrete locally finite graph $E$. It follows that $G^0$ is also discrete. It is known that in this case $C_0(E^0)\rtimes G$ is strongly Morita equivalent with a commutative $C^*$-algebra $C_0(X)$ with $X$ at most countable.

If we denote by $\{p_x\}_{x\in X}$ the minimal projections in $C_0(X)$, recall that
the isomorphism classes of separable nondegenerate $C^*$-correspondences ${\mathcal H}$ over $C_0(X)$ determine a discrete graph. More precisely, the  $*$-homomorphism $\phi: C_0(X)\to {\mathcal L}({\mathcal H})$ will define an incidence matrix $(a_{xy})_{x,y\in X}$ where $a_{xy}=\dim \phi(p_x){\mathcal H}p_y.$
(see Theorem 1.1 in \cite{KPQ}).

We recall now the following result from \cite{MPT}:

\begin{lem} Suppose $C$ and $D$ are strongly Morita equivalent $C^*$-algebras with $C$--$D$ imprimitivity bimodule ${\mathcal Z}$.

 If ${\mathcal H}$ is a $C^*$-correspondence over $C$, then ${\mathcal K}={\mathcal Z}^*\otimes_C{\mathcal H}\otimes _C{\mathcal Z}$ is a $C^*$-correspondence over $D$ such that ${\mathcal O}_{\mathcal H}$ and ${\mathcal O}_{\mathcal K}$ are strongly Morita equivalent.

\end{lem}

 \begin{thm}  Let $E$ be  a locally finite discrete  graph  with no sources and let $G$ be a  groupoid with compact isotropy groups acting on $E$. Then the crossed product $C^*(E)\rtimes G$ is  strongly Morita equivalent to a graph $C^*$-algebra, where the number of vertices is the cardinality of the spectrum of $C_0(E^0)\rtimes G$.
\end{thm}
\begin{proof}
The idea is to decompose the  $C^*$-correspondence ${\mathcal H}_E\rtimes G$ over the  $C^*$-algebra $C_0(E^0)\rtimes G$, which is strongly Morita equivalent with a commutative $C^*$-algebra.

Let $\ds C_0(E^0)\rtimes G\cong \bigoplus_{i=1}^NM_{n(i)}$, where $N\in {\mathbb N}\cup\{\infty\}$ and $M_{n(i)}$ denotes the set of $n(i)\times n(i)$  matrix algebras. Consider now the  graph with  $N$ vertices   and at each vertex $v_i$ we assign the $C^*$-algebra $M_{n(i)}$. If $p_i$ is the unit in $M_{n(i)}$, whenever $p_i({\mathcal H}_E\rtimes G)p_j\neq 0$, we decompose this as a direct sum of minimal $M_{n(i)}$--$M_{n(j)}$ $C^*$-correspondences. Since $\CC$--$\CC$ $C^*$-correspondences are Hilbert spaces, a  minimal $M_{n(i)}$--$M_{n(j)}$ $C^*$-correspondence  is of the form $M_{n(i),n(j)}$, the set of rectangular matrices with $n(i)$ rows and $n(j)$ columns, with the obvious bimodule structure and inner product.
This decomposition determines  the number of edges between $v_j$ and $v_i$.

Since in this case $J_{\Hh_E}=C_0(E^0)$, we deduce that \[\Oo_{\Hh_E\rtimes G}\cong \Oo_{\Hh_E}\rtimes G\cong C^*(E)\rtimes G.\] It follows that $C^*(E)\rtimes G$ is isomorphic to the $C^*$-algebra of a graph of (minimal) $C^*$-correspondences  (see \cite{DKPS}), hence strongly Morita equivalent to a graph $C^*$-algebra. 

Note  that in  this case $C^*(E)$ is the $C^*$-algebra of a groupoid $\Gg_E$ and that $C^*(E)\rtimes G=C^*(\Gg_E\rtimes G)$

\end{proof}
\begin{cor}
If $G$ acts on $E$ as above, then $G$ acts on the AF-core $C^*(E)^\TT$ and $C^*(E)^\TT\rtimes G\cong (C^*(E)\rtimes G)^\TT$ is an AF-algebra.
\end{cor}
\bigskip

\section{The $C^*$-correspondence of a groupoid representation}

\bigskip

Given a  group $G$ and a unitary representation $\rho:G\to U({\mathcal H})$ on a Hilbert space ${\mathcal H}$, in \cite{D1} we constructed a $C^*$-correspondence ${\mathcal E}(\rho)={\mathcal H}\otimes_{\mathbb C} C^*(G)$ over $C^*(G)$ and we studied the Cuntz-Pimsner algebra ${\mathcal O}_{{\mathcal E}(\rho)}$. 


Now, let $G$ be a locally compact groupoid with a Haar system $\lambda$. Given a Hilbert bundle ${\mathcal H}$ over the unit space $G^0$ and a representation $\rho:G\to$Iso$({\mathcal H})$, denote by $\Gamma_c(G,r^*{\mathcal H})$ the space of continuous sections with compact support of the pull-back bundle $r^*{\mathcal H}$. We define the left and right multiplications of $C_c(G)$ on $\Gamma_c(G,r^*{\mathcal H})$ as in \cite{Re} by
\[(f\xi)(g)=\int_Gf(h)\rho(h)\xi(h^{-1}g)d\lambda^{r(g)}(h),\]
\[(\xi f)(g)=\int_G\xi(gh)f(h^{-1})d\lambda^{s(g)}(h),\]
and the inner product
\[\la\xi,\eta\ra(g)=\int_G\la\xi(h^{-1}),\eta(h^{-1}g)\ra_{s(h)}d\lambda^{r(g)}(h)\]
for $f\in C_c(G)$ and $\xi, \eta\in \Gamma_c(G,r^*{\mathcal H})$.
The completion ${\mathcal E}(\rho)$ of $\Gamma_c(G,r^*{\mathcal H})$ becomes a $C^*$-correspondence over $C^*(G)$, where the left action of $C_c(G)$ extends to a $*$-homomorphsim $C^*(G)\to {\mathcal L}({\mathcal E}(\rho))$.

Given representations $\rho_i:G\to$Iso$({\mathcal H}_i)$ for $i=1,2$, one may consider the tensor product representation $\rho_1\otimes\rho_2:G\to$Iso$({\mathcal H}_1\otimes{\mathcal H}_2)$ and the $C^*$-correspondence of $\rho_1\otimes\rho_2$ is the composition of the $C^*$-correspondences for $\rho_1,\rho_2$, see Theorem 3.4 in \cite{Re}.

We plan to study the Cuntz-Pimsner algebras $\Oo_{\Ee(\rho)}$ in future work.
\begin{rmk} In the context of self-similar actions of a groupoid $G$ on the path space of a finite graph $E$, there are other $C^*$-correspondences over $C^*(G)$ which were studied recently. Laca, Raeburn, Ramagge and Whittaker in \cite{LRRW} characterized KMS states on the Toeplitz algebra ${\mathcal T}_{\mathcal M}$ and on the Cuntz-Pimsner algebra ${\mathcal O}_{\mathcal M}$, where \[{\mathcal M}={\mathcal H}_E\otimes_{C(E^0)}C^*(G)\] becomes a Hilbert $C^*(G)$-module with usual inner product and right multiplication.
The left multiplication is determined by the unitary representation \[\rho:G\to {\mathcal L}({\mathcal M}),
\;\rho(g)(e\otimes f)=\begin{cases}(g\cdot e)\otimes \delta_{g\mid_e}f\;\text{if}\; s(g)=r(e)\\0\hspace{10mm}\text{otherwise,}\end{cases}\]
where $e\in E^1, \;f\in C^*(G)$ and $\delta_h\in C_c(G)$ denotes the point mass.
\end{rmk}


\bigskip

\begin{thebibliography}{0000}

\bigskip

\bibitem{AR} 
C. Anantharaman-Delaroche and J. Renault, {\em Amenable groupoids}, 
L'Enseignement Math\'ematique 36,
Gen\`eve, 2000.  


\bibitem{Bl} E. Blanchard, {\em D\' eformations de $C^*$-alg\` ebres de Hopf}, Bull. Soc. Math. France 124 (1996), 141--215.

\bibitem{Br} R. Brown, {\em Groupoids as coefficients}, Proc. London Math. Soc. (3) 25(1972) 413--426.

\bibitem{BM} A. Buss, R. Meyer, {\em Iterated crossed products for groupoid fibrations}, arXiv: 1604.02015v1.

\bibitem{D} V. Deaconu, {\em Group actions on graphs and $C^*$-correspondences}, to appear in Houston J. Math.

\bibitem{D1} V. Deaconu, {\em Cuntz-Pimsner algebras of  group representations}, submitted.

\bibitem{DKPS}  V. Deaconu, A. Kumjian, D. Pask, A. Sims, {\em Graphs of $C^*$-correspondences and Fell bundles},  Indiana Univ. Math. J. 59 (2010), no. 5, 1687--1735.

  \bibitem{DKQ} V. Deaconu, A. Kumjian, J. Quigg, {\em Group actions on topological graphs}, Ergod. Th. Dyn. Sys. 32 (2012), no. 5, 1527--1566.
 \bibitem{DKR} V. Deaconu, A. Kumjian, B. Ramazan, {\em Fell bundles associated to groupoid morphisms},  Math. Scan. vol. 102 no.2 (2008), 305--319.
  
  
 \bibitem{DR1}  S. Doplicher, J.E. Roberts, {\em Duals of compact Lie groups realized in the Cuntz algebras and their actions on C*-algebras}, J. of Funct. Anal.  74(1987) 96--120.
 
 \bibitem{DR2} S. Doplicher, J.E. Roberts, {\em A new duality theory for compact groups}, Invent. Math. 98 (1989) 157--218.
 
 
 
  
  
  
  

   \bibitem{HN}  G.~Hao, C.-K. Ng, \emph{Crossed products of C*-correspondences by
  amenable group actions}, J. Math. Anal. Appl. 345 (2008), no.~2,
  702--707.
    
    
    \bibitem{KPW} T. Kajiwara, C. Pinzari, Y. Watatani, {\em Ideal structure and simplicity of the $C^*$-algebras generated by Hilbert bimodules}, J. Funct. Anal. 159 (1998), no. 2, 295--322.
    
  \bibitem{KPQ} S. Kaliszewski, N. Patani, J.Quigg, {\em Characterizing graph $C^*$-correspondences}, Houston J. Math. 38 (2012), no. 3, 751--759.
    
 \bibitem{K} T. Katsura, {\em On $C^*$-algebras associated with $C^*$-correspondences},    J. Funct. Anal. 217(2004), 366--401.
    
    

    
    
    
      
  
  \bibitem{Ku} A. Kumjian, {\em On equivariant sheaf cohomology and elementary $C^*$-bundles}, J. Operator Theory 20(1988), 207--240.
  
  \bibitem{KMRW} A. Kumjian, P. Muhly, J. Renault, D. Williams, {\em The Brauer group of a locally compact groupoid}, American J. Math. 120(1998), 901--954.
  
  \bibitem{KP}  A. Kumjian, D. Pask, {\em $C^*$-algebras of directed graphs and group actions}, Ergod. Th. Dyn. Sys. 19(1999) 1503--1519.
  
  
  
       




\bibitem{LRRW} M. Laca, I. Raeburn, J. Ramagge, M.F. Whittaker, {\em Equilibrium states on operator algebras associated to self-similar actions of groupoids on graphs}, arXiv preprint 1610.00343v1.

\bibitem{LR} N.P. Landsman, B. Ramazan, {\em Quantization of  Poisson algebras associated to Lie algebroids}, Groupoids in
Analysis, Geometry, and Physics, Arlan Ramsay and Jean Renault
(editors), AMS Contemporary Mathematics 282 (2001),  159--192. 

\bibitem{L1} P.-Y. Le Gall, {\em Th\' eorie de Kasparov \' equivariante et groupo\" ides}, Th\' ese de Doctorat, Universit\' e Paris VII, 1994.

\bibitem{L2} P.-Y. Le Gall, {\em Th\' eorie de Kasparov \' equivariante et groupo\" ides I}, K-Theory, 16 (1999),
361--390.

 
 \bibitem{M} P.D. Mitchener, {\em C*-categories, groupoid actions, equivariant KK-theory, and the Baum-Connes conjecture}, J. Funct. Anal. 214 (2004) 1--39.

\bibitem{MPT} P. Muhly, D. Pask,  M. Tomforde, {\em Strong shift equivalence of
$C^*$-correspondences},  Israel J. Math. 167 (2008), 315--346.

\bibitem{MRW1} P.S. Muhly, J. Renault and D. Williams, {\em Equivalence and isomorphism for groupoid C*-algebras}, J. Operator Theory 17(1987) 3--22.
  
  \bibitem{MW} P. Muhly, D. Williams, {\em Renault's equivalence theorem for groupoid crossed products}, NYJM Monographs vol. 3, 2008.


    \bibitem{P2} M. V. Pimsner, \emph{A class of $C^*$-algebras
    generalizing both Cuntz-Krieger algebras and crossed products by
    $\mathbb{Z}$}, Fields Inst. Comm. 12 (1997),
    189--212.

\bibitem{Re}  J. Renault, {\em Groupoid cocycles and derivations},  Ann. Funct. Anal. 3 (2012), no. 2, 1--20. 
\bibitem{V} E. Vasselli, {\em The $C^*$-algebra of a vector bundle and fields of Cuntz algebras}, J. Funct. Anal. 222 (2005), no. 2, 491--502.


\bibitem{W} D. Williams, {\em Crossed-Products of $C^*$-algebras}, Mathematical Surveys and Monographs, vol. 134, AMS, Providence, RI, 2007.


\end{thebibliography}
\end{document}